\documentclass[12 pt]{amsart}
\usepackage{amsmath,amssymb,amsthm}
\usepackage{longtable}
\usepackage{a4wide}
\usepackage[usenames,dvipsnames]{xcolor}

\def \wh {\widehat}

\def \Mat{\mbox {\rm Mat}}

\def \Fi {\mbox {\rm Fi}}

\def \J{\mbox {\rm J}}

\def \B{\mbox {\rm B}}

\def \McL{\mbox {\rm McL}}

\newcommand{\PSL}{\mathrm{PSL}}
\newcommand{\SL}{\mathrm{SL}}
\newcommand{\PGL}{\mathrm{PGL}}\newcommand{\PSU}{\mathrm{PSU}}

\def \syl {\hbox {\rm Syl}}\def \Syl {\hbox {\rm Syl}}
\def \ov{\overline}

\def \PSp {\mathrm {PSp}}

\newtheorem{defn}{Definition}[section]
\newtheorem{theorem}[defn]{Theorem}
\newtheorem{remark}[defn]{Remark}
\newtheorem{lemma}[defn]{Lemma}

\newtheorem{prop}[defn]{Proposition}
\def \d {\mathrm d}

\def\B {\mathrm B}

\newcounter{claim}[defn]

\title{The diameter of the commuting graph of a finite group with trivial centre}
\author{G. L.  Morgan}\author{C. W.  Parker}

\address{G. L.  Morgan and C. W. Parker \\
School of Mathematics\\
University of Birmingham\\
Edgbaston\\
Birmingham B15 2TT\\
United Kingdom} \email{morganl@maths.bham.ac.uk, c.w.parker@bham.ac.uk}

	\email{ }

\begin{document}

\begin{abstract} The commuting graph $\Gamma$ of a finite  group with trivial centre is examined. It is shown that the connected components of  $\Gamma$ have diameter at most $10$.\end{abstract}
\maketitle

\section{Introduction}

Suppose that $G$ is a    group.  The \emph{commuting graph}  $\Gamma(G)$ of $G$ is the graph which has vertices the non-central elements of $G$ and two distinct vertices of $\Gamma(G)$ are adjacent if and only if they commute in $G$. In this paper we study the commuting graph of a finite group. In \cite{I}  Iranmanesh and  Jafarzadeh demonstrate that the commuting graph of  $\mathrm{Sym}(n)$ and $\mathrm {Alt}(n)$, the symmetric group and alternating group on $n$ letters,  is either disconnected or has diameter  at most $5$ and, in the same article, they conjecture that there is an absolute upper bound for the diameter of a connected commuting graph of a non-abelian finite group. This conjecture was shown to be incorrect in \cite{parkergiudici} where  an infinite family of special $2$-groups with commuting graphs of increasing diameter was presented. However, the core suggestion of the conjecture is not far from the mark. In this paper we prove

\begin{theorem}\label{MainTheorem} Suppose that $G$ is a finite group with trivial centre. Then every connected component of the commuting graph of $G$ has diameter at most $10$. In particular, if the commuting graph of $G$ is connected, then its diameter is at most $10$.
\end{theorem}

Evidence that a theorem such as Theorem~\ref{MainTheorem} must be true has been growing for over a decade.
In 2002,  Segev and Seitz  showed   that the commuting graph of a classical simple group defined over a field of order greater than $5$ is either disconnected or has  diameter at most $10$ and at least $4$ \cite[Corollary (pg. 127), Theorem 8]{SS}. They also proved that the commuting graph of the simple exceptional Lie type groups other than $\mathrm E_7(q)$ and the sporadic simple groups are disconnected \cite[Theorem 6]{SS}. In 2008,  as  remarked above,  Iranmanesh and  Jafarzadeh demonstrated that  the commuting graph of  the almost simple groups $\mathrm{Sym}(n)$ and $\mathrm {Alt}(n)$ is either disconnected or has diameter  at most $5$ \cite{I}.  In 2010, Giudici and Pope \cite{G1} modified the approach of Segev and Seitz \cite{SS} and showed that for  $G=\mathrm{GL}_n(2)$, $n \ge 5$,  if $\Gamma(G)$ is connected then it has diameter at most 8 and, by computing specific commuting matrices, they demonstrate that for $G=\mathrm{GL}_n(r)$, $r$ an odd prime, if $\Gamma(G)$ is connected then it has diameter at most 6.
More recently in 2012, Giudici and Pope \cite{G}  investigated some further special cases  and showed that various kinds of finite groups had bounded diameter commuting graph. They also constructed a soluble group with commuting graph of diameter $6$. This was  followed  by the second author of this paper in  \cite{parker}  in which it is proved that the commuting graph of a finite soluble group with trivial centre  is either disconnected or has diameter at most 8. Furthermore in \cite{parker} finite soluble groups with trivial centre and connected commuting graph of diameter 8 are presented.

Suppose that $G$ is a non-abelian finite group. The \emph{prime graph} of $G$, $\pi(G)$,  has vertex set the set of primes dividing $|G|$, and two primes $r$ and $s$  are connected if and only if $G$  has an element of order $rs$.
The prime graph was introduced by Gruenberg and Kegel in  \cite{GK} where they proved  the first results about its  connectivity. In particular, they partly described the structure of groups with disconnected prime graph and showed that the soluble groups with disconnected prime graph are either Frobenius groups or $2$-Frobenius groups.  The non-abelian simple groups with disconnected prime graph were determined by Williams \cite{Williams} and  Iiyori and Yamaki \cite{Iiyori} (see also Kondrat\'ev \cite{Kondratev}) and Lucido \cite{Lucido} worked out which of the almost simple groups have connected prime graphs. Here we mention that a long paper by Suzuki \cite{Suzuki} investigates the prime graph using methods from the odd order theorem.  It turns out that in finite groups $G$ with trivial centre the $G$-classes of connected components of the commuting  graph are in one to one correspondence  with the  connected components of the prime graph (see Theorem~\ref{primegraph} and \cite[Lemma 4.1]{I}). In particular, $\Gamma(G)$ is connected if and only if $\pi(G)$ is connected and consequently the almost simple groups with connected commuting graphs are known because of the work of Williams \cite{Williams} and Lucido \cite{Lucido}.

Our proof of Theorem~\ref{MainTheorem} first reduces the problems to one about the almost simple groups. This is quickly accomplished in Section~\ref{almostsec} and leaves the harder task of bounding the diameter of a connected component in an almost simple group.  To achieve this we take advantage of the work of Williams \cite{Williams} (see Theorem~\ref{Williams}) where he studied the connectivity of the prime graph. The result is that all the connected components of the commuting graph of a non-soluble group which do not contain elements of order $2$ have diameter one.
Our plan of  attack for determining the diameter of the commuting graph of the finite almost simple groups is as follows. For the sporadic simple groups we use the character tables and for the symmetric and alternating groups an elementary lemma, Lemma~\ref{involutionsgood}, about the distance apart of involutions in $\Gamma(G) $  in groups with at least two conjugacy classes of involutions (they are at most $3$ apart) can be deployed to show that the diameter of a connected component is at most $8$. In the case that $\Gamma$ is connected, we could use \cite{I} to get that the diameter is at most $5$. The meat of the paper is to address the problem of bounding the diameter of  the connected components of the commuting graphs of the almost simple groups of Lie type defined in characteristic $r$. The linear groups $\PSL_2(r^a)$ together with the Suzuki-Ree groups require special arguments. For the remaining groups we adapt and  extend the arguments of Segev and Seitz  \cite{SS} to the case where  the commuting graph of $G$ with $F^*(G)$ a simple group of Lie type in characteristic $r$ is possibly disconnected. The theorem follows from the fact that any two long root elements have distance at most $2$ apart and every element of $G$ has distance at most $4$ from a long root element.

 It is beyond doubt that the bounds on the diameter of $\Gamma(G)$ that we obtain for the almost simple groups can be improved upon in many cases. However as the bounds we obtain  suffice to prove Theorem~\ref{MainTheorem}, we have chosen not to pursue more accurate estimates.  It would be interesting to know the actual values for the diameters of the connected components commuting graphs of the sporadic simple groups for example. We also mention that the commuting graph of  ${}^2\B_2(2^5){:}5$ has a connected component which has diameter at least 8.

In a different direction, Solomon and Woldar \cite{SolomonWoldar} have recently proved that if $S$ is a finite simple group and $G$ is an arbitrary group with $\Gamma(S) \cong \Gamma(G)$, then $S \cong G$. Thus the commuting graph of a simple group, determines the simple group uniquely.

The paper is organized as follows. In Section~\ref{preliminaries} we collect some background results on Frobenius groups which are required in Section~\ref{almostsec} and groups of Lie type that are frequently used in Section~\ref{Lie}. In Section \ref{prime}, we establish some elementary results about the commuting graph and prove the aforementioned theorem relating the commuting graph and the prime graph. Section~\ref{isolatedsec} is devoted to  the study of isolated subgroups, their relationship with strongly $p$-embedded subgroups and in particular we modestly improve the statement of \cite[Theorem 3]{Williams} in Theorem~\ref{Williams}. Finally, in Section~\ref{isolatedsec}, we  show that in a non-abelian simple group, isolated subgroups are for the most part  cyclic.   In Section~\ref{almostsec}, we prove the reduction theorem to simple groups. The next  three  sections establish the theorem for the finite simple groups. In a short final section we prove the main theorem.

Our group theoretic notation is mostly standard and follows that in \cite{Aschbacher, Gor}. In particular we mention that for a group $G$, $G^\#$ is the set of non-identity elements of $G$.  For $x \in G$, the conjugacy class of $x$ in $G$ is denoted by $x^G=\{x^g\mid g \in G\}$. If  $\{x,y\}$ is an edge in $\Gamma(G)$ or, if $ x=y$, then we write $x \sim y$. In particular, $x \sim y$ indicates that $x,y \in G \setminus Z(G)$. If $x $ and $ y $ are vertices in $\Gamma(G)$, then $\mathrm d(x,y)$ denotes the distance between $x$ and $y$. We apply similar convenient conventions to the prime graph of $G$. Thus, for primes $r,s$, $r \sim s$ means that either $r=s$ or there is an element of $G$ of order $rs$.  For subsets of vertices  $\mathcal S$ of $\Gamma(G)$ we define $$\d(x,\mathcal{S})=\min\{\d(x,r)\mid r \in \mathcal{S}\}.$$

\section{Preliminaries}\label{preliminaries}

In this section, we provide the fundamental results that we need from \cite{GLS3} and \cite{Hu}.

\begin{lemma}\label{meta} Suppose $X$ is a Frobenius complement. Then every Sylow subgroup of $X$  is cyclic or
generalized quaternion. If $X$  has odd order then any  two elements of   prime order  commute and  $X$ is metacyclic.  If $X$ has even order, then $X$ contains a unique involution.
\end{lemma}
\begin{proof}
The Sylow subgroups of $X$ are described in \cite[ Satz 8.15, p. 505]{Hu}. Suppose that $X$ is soluble. Let $r$ and $s$ be primes dividing $|X|$ and $a$ and $b$ be elements of order $r$ and $s$ respectively. If $X$ has odd order then $a$ and $b$ commute by \cite[ Satz 8.16 (b), p. 506]{Hu}. If $X$ has even order, see \cite[Satz 8.18 (a), p.506]{Hu}.
\end{proof}

We will also need the following well-known lemmas. Their easy proofs can be found in \cite{parker} (as well as many other places).

\begin{lemma}\label{Frob} Suppose that $X$ is a group and $X= JK$ with $J$ a proper normal subgroup of $X$ and $K$ a complement to $J$. Then  $C_X(k) \le K$ for all $k \in K^\#$ if and only if  $X$ is a Frobenius group.\qed
\end{lemma}

\begin{lemma}\label{Frob2} Suppose that $J$ is a proper normal subgroup of $X$. Then $X$ is a Frobenius group if and only if $C_X(j) \le J$ for all $j \in J^\#$.\qed
\end{lemma}

In Section \ref{Lie} we  require information about centralisers of automorphisms in Lie type groups. We adopt notation and terminology from \cite{GLS3}. In particular, for the description of the types of automorphism see \cite[2.5.13]{GLS3} and \cite[4.1.8]{GLS3}.  A pivotal role in our demonstration that the diameter of the commuting graphs of the Lie type groups defined in characteristic $r$ have small diameter is played by special elements which are called \emph{long root} elements.  Suppose that $K$ is a finite simple group of Lie type which is not a Ree or Suzuki group.  We follow \cite[Definition 2.2.1]{GLS3} and let $(\overline K,\sigma)$ be a $\sigma$-setup for $K$; here $\ov K$ is an adjoint simple algebraic group and $\sigma$ is a Steinberg endomorphism of $\ov K$.  See \cite[Example 3.2.6]{GLS3} for fuller version of the following discussion.
 We know $K=O^{r'}( C_{\ov K}(\sigma))$.  Let $\ov B$  and $\ov T$ be a $\sigma $-invariant  Borel subgroup of $\ov K$ and Torus of $\ov K$ contained in $\ov B$ respectively.  Also let $\Sigma$ be the root system of $\ov K$  (with respect to $\ov B$ and $\ov T$)
with fundamental system $\Pi$ and $\rho$ be the highest root of $\Sigma$ with respect to $\Pi$.  Then $P_\rho= N_{\ov K}(\ov X_\rho) \ge \ov B$ is a parabolic subgroup of $\ov K$. Furthermore, as $\ov B$ and $\ov T$ are $\sigma$-invariant and $\sigma$ does not permute root lengths (otherwise $K$ would be a Ree or Suzuki group), we see that $\ov X_{\rho}$ is $\sigma$-invariant.  We have $X_{\rho} = C_{\ov{X_\rho}}(\sigma)$ is an elementary abelian group of order $r^a= q(\ov K, \sigma)$ (again see \cite[Theorem 2.1.12]{GLS3}).  The  non-trivial elements  of $X_{\rho}$ are  \emph{long root elements}. In \cite[Example 3.2.6 and Proposition 3.2.9]{GLS3} the subgroups of $K$ which are $K$-conjugate to $X_{\rho}$ are called \emph{long root subgroups}. We will adhere to this nomenclature.

 Alternative definitions of  long root elements are given in other sources and this is the primary reason for the above discussion. Of course the set of elements is the same. For example, the structure of the ``twisted" root subgroups in $K$ are described in \cite[Theorem 2.4.1]{GLS3}. If $X_{\hat \alpha}$ is a ``long root subgroup" (which now need not be abelian), then the non-trivial elements of $Z(X_{\hat \alpha})$ are called long root elements of $K$.

 Let $S = \langle X_\rho, X_{-\rho}\rangle$. Then, so long as $K \not \cong \PSL_2(r^a)$ and $K$ is not a Ree or Suzuki group,  $S \cong \SL_2(r^a)$  and the $K$-conjugates of $S$ are called \emph{fundamental $\SL_2(r^a)$-subgroups} of $K$ \cite[Definition 3.2.7]{GLS3}.

\begin{lemma}\label{rootgroups}
Suppose  $K$ is a simple group of Lie type defined in characteristic $r$  such that $K \not \cong \PSL_2(r^a)$ and $K$ is not a Ree or Suzuki group. Then
\begin{enumerate}
\item  $K$ has one conjugacy class of long root subgroups;
\item  The  conjugacy class of long root subgroups of $K$ is invariant under $\mathrm{Aut}(K)$ unless $K \cong \mathrm B_2(2^a) \cong \mathrm {Sp}_{4}(2^a)$, $\mathrm F_4(2^a)$ or $\mathrm G_2(3^a)$.
\item   Either $K$ has exactly one conjugacy class of long root elements or $K\cong \mathrm{PSL}_2(r^a) \cong \PSp_2(r^a)$ or $\mathrm C_{n}(r^a)\cong \mathrm{PSp}_{2n}(r^a)$, $n \ge 2$, with $r$ odd and $K$ has exactly two classes of long root elements.
\item If $x, y \in K$ are long root elements, then either $\langle x, y \rangle $ is an $r$-group or $\langle x,y \rangle$ is contained in a fundamental $\SL_2(r^a)$ of $K$.
\item The centralizer of a fundamental $\SL_2(r^a)$ is non-trivial unless $K \cong \PSL_3(2)$  or $K\cong \PSL_3(4)$. Furthermore, if  the root system of $\overline K$ is not $\mathrm A_2$, then $S$ commutes with an element of order $r$.
\end{enumerate}
\end{lemma}

\begin{proof} The first three points are discussed in \cite[Example 3.2.6]{GLS3}. However, there they only state that $C_{\ov K}(\sigma)$ acts transitively on the long root subgroups of $K$. None-the-less, by \cite[Theorem 2.2.6 (g)]{GLS3}, $C_{\ov K}(\sigma) = KC_{\ov T}(\sigma)$ and $C_{\ov T}(\sigma)$ normalizes $X_{\rho}$. Hence $K$ has one conjugacy class of long root subgroups. So (i) holds.

Parts (ii) and (iii) are explained in \cite[Page 103, paragraph 3 and paragraph -3]{GLS3}.

For part (iv) we let $S$ be a fundamental $\SL_2(r^a)$.  If $r$ is odd, then $Z(J) \ne 1$ and we are done. So assume that $r=2$.  Then, by \cite[Theorem 3.2.8]{GLS3} $N_K(S)= C_K(S)S$ and,   in addition, setting $J= -\rho ^\perp \cap \Pi$, the group $\langle X_\tau \mid \tau \in \pm J\rangle$ centralizes $S$. Therefore, as long as $\Sigma$ is not $\mathrm A_2$, we have $C_K(S) \ne 1$.  If $\Sigma = \mathrm  A_2$, then $K$ is either $\PSL_3(r^a)$ or $\PSU_3(r^a)$ and  the corresponding  universal groups contain subgroups isomorphic to $\mathrm{GL}_2(r^a)$ and $\mathrm{GU}_2(r^a)$. These subgroups project to subgroups of $K$ containing conjugates of $S$. Especially the centres of these subgroups have order $r^a-1$ and $r^a+1$ respectively. Thus in $K$, $S$ has a non-trivial centralizer whenever $r^a-1 > 3$ when $K \cong \PSL_3(r^a)$ or $r^a+1 >3$ when $K \cong \PSU_3(r^a)$. Since $\PSU_3(2)$ is soluble, we have proved our result.
\end{proof}

\begin{lemma}
\label{lem:lietypecents}
Let $K$ be a group of Lie type defined over a field of characteristic $r$ and $x$ be an automorphism of $K$ of prime order $p$ with $p\neq r$. Write $K={}^d\Sigma(q)$,  $C=C_K(x)$, $L=O^{r'}(C)$ and let $Z$ be the kernel of the covering $K_u \rightarrow K$ ($K_u$ the universal group). The following hold.
\begin{itemize}
\item[(a)] If $x$ is a field automorphism, then $L \cong {}^d\Sigma(q^{\frac{1}{p}})$  and $L$ contains a long root element of $K$.
\item[(b)] If $x$ is a graph-field automorphism, then $d=1$, $p=2$ or $3$ and $L\cong {}^p\Sigma(q^{\frac{1}{p}})$.
\item[(c)] If $x$ is an inner-diagonal or graph automorphism, then $L$ is a central product of groups of Lie type defined over fields of characteristic $r$,  there is an abelian $r'$-subgroup $T$ such that $C/LT$ is an elementary abelian $p$-group isomorphic to a subgroup of $Z$.
\item[(d)] If $x$ is of equal-rank type, then $|C_C(L)/(C \cap Z(K))|\leq p$.
\item[(e)] If $x$ is of graph type, then $C_C(L) \leq Z(K)$.
\end{itemize}
\end{lemma}
\begin{proof}
For (a) and (b) see \cite[Theorem 4.9.1]{GLS3} and for the remaining parts see \cite[Theorem 4.2.2]{GLS3}.
\end{proof}

\section{The   commuting graph and the prime graph}\label{prime}

For the remainder of the paper, we suppose that all groups are finite groups.  Suppose that $G$ is a   group and let $\pi(G)$ be the prime graph of $G$ and $\Gamma=\Gamma(G)$ be the commuting graph of $G$. We start this section by presenting a series of  elementary lemmas detailing various properties of the commuting graph before moving on to the relationship between the commuting graph and the prime graph. Throughout this section $G$ is a   group with $Z(G)=1$.

\begin{lemma}\label{primeelements} Suppose  $x, y \in G^\#$ and that $x$ and $y$ are joined by a path of length $m$ in $\Gamma$. Then there exist    elements $x_2, x_3, \dots, x_{m-1}$ of prime order such that $$x= x_1 \sim x_2 \sim \dots \sim x_{m-1} \sim x_m=y$$ is a path in $\Gamma$. In particular, if $d$ is the maximum distance in $\Gamma$ between two elements of prime order, then $\Gamma$ has diameter at most $d+2$.
\end{lemma}

\begin{proof} Let $x= x_1 \sim x_2^*\sim  \dots \sim x_{m-1} ^*\sim x_m=y$ be a path in $\Gamma$ connecting $x$ to $y$. For $2\le i \le m-1$, let  $x_i$ be an element of prime order in $\langle x_i^*\rangle$. Then, as $Z(G)=1$, $$x= x_1 \sim x_2\sim  \dots \sim x_{m-1} \sim x_m=y.$$
\end{proof}

We use Lemma~\ref{primeelements} without further reference. Thus often the paths we consider will be between elements of prime order and will consist of elements of prime order and always the ``inner elements" of a path will have prime order.

\begin{lemma}\label{normal} Suppose that $\Psi$ is a connected component of $\Gamma$. If $\Psi$ contains a conjugacy class of $G$, then $\Psi$ is a normal subset of $G$.
\end{lemma}
\begin{proof}   Let $x \in \Psi$  and  $\mathcal C$  be a conjugacy class of $G$ in $\Psi$.  Then there exists a path in $\Psi$  joining $x$ to an  element $y \in \mathcal C$.  Hence, for all $g\in G $,  $x^g$  is connected to $ y^g \in \mathcal C \subseteq \Psi$. Hence $x^g \in \Psi$ and so $\Psi$ is a normal subset of $G$.
\end{proof}

 The next lemma will be used when we investigate the commuting graph of the almost simple groups of Lie type and  will be applied to the class of long root elements of Lie type groups  once we have shown that any two such elements have distance $2$ apart. This then explains one of the places where our general bound of 10 on the diameter of the commuting graph emerges.

\begin{lemma}\label{outside}
Suppose $K$ is a normal subgroup of $G$, $a \in G^\#$ with $a^G= a^K$. Then, for every $x \in G\setminus K$,  we have $\d(x,a^G) \le 4$.
\end{lemma}

\begin{proof} Since $a^G= a^K$ and $K$ is normal in $G$, $G= C_G(a)K$.  Let $x \in G \setminus K$ and assume that $p$ divides the order of $xK$ as an element of $G/K$.  Let $x^*$ be a generator of the Sylow $p$-subgroup of $\langle x \rangle$, $P_0 \in \Syl_p(C_G(a))$ and $P \in \syl_p(G)$ with $P_0 \le P$.  Since $p$ divides $|G/K|$, $P_0$ is non-trivial. Let $g \in G$ be such that $x^*\in P^g $ and  let $y\in Z(P^g)^\#$. Then $y$ commutes with $P_0^g \le C_G(a^g)$. Hence there exists $z \in P_0^g$ such that $$x\sim x^* \sim y \sim z \sim a^g$$ and this proves the claim.
\end{proof}

\begin{remark} There are elements of order $205= 5\cdot 41$ in $G=\mathrm {PGU}_5(4)$ which are not contained in $\PSU_5(4)$ and which have distance $4$ from root elements (and from involutions). This shows that Lemma~\ref{outside} cannot be sharpened.
\end{remark}

The following observation can in some instances be used to give better bounds on the diameter of $\Gamma$  compared to an approach using long root elements.  In particular, this is the case when we cannot show that root elements are close together. We also apply this lemma when considering the alternating groups and the sporadic simple groups.

\begin{lemma}\label{involutionsgood}
Suppose that $G$ has at least two conjugacy classes of involutions and let $\mathcal{I}$ be the set of involutions in $G$. Then $\mathcal{I}$ is a connected  subgraph of $\Gamma $ and has diameter at most 3. Moreover, there is a unique connected component containing all the elements of even order in $G$.
\end{lemma}

\begin{proof}  Suppose first that $x, y \in \mathcal I$. Let $T_y $ be a Sylow $2$-subgroup of $G$ containing $y$, if possible, chosen so that $y \in Z(T_y)$. Then, as $G$ has two conjugacy classes of involutions,  there exists  $z\in \mathcal I \cap T_y$ such that $z$ is not conjugate to $x$. If we can choose $z= y$, then we do so. If $z$ cannot be chosen to be $y$, then $x$ and $y$ must be conjugate. If $y\not\in Z(T_y)$, then $x$ is not conjugate to an element of $Z(T_y)$ by our choice of $T_y$. So we choose $z \in Z(T_y)^\#$. If  $y \in Z(T_y)$,  we just choose $z \in T_y$ not conjugate to $x$.  As $x$ and $z$ are not conjugate,    $\langle x,z \rangle$ is a dihedral group of order divisible by $4$ and so there exists an involution  $w \in Z( \langle z,x\rangle)^\#$. Therefore $$x \sim  w \sim z\sim y$$ is a path in $\Gamma $ consisting just of elements of $\mathcal I$. Hence any two members of $\mathcal I$   are at distance at most $3$ apart.
Since every element of even order is connected to an involution, all elements of even order  can be  connected in $\Gamma$ by a path of length at most $5$.
\end{proof}

 The problem with Lemma~\ref{involutionsgood} is that   when applied in conjunction with Lemma~\ref{outside} it gives a bound on the diameter of $\Gamma$ of $11$ whereas we would like to demonstrate the diameter is at most  10. The next lemma is the basis of our technical solution to this problem.

\begin{lemma}\label{outside2}
Suppose $p$ is a prime, $K$ is a normal subgroup of $G$, $a \in G^\#$ and $x \in G \setminus K$. Set $G_0= \langle x \rangle K$.  Assume that $f \in C_{G_0}(a) $ is a $p$-element, $x_p \in \langle x\rangle \cap K$ has order $p$ and $C_K(x_p)$ is cyclic.  Then $\d(x,a^G) \le 3$.
\end{lemma}

\begin{proof} Let $P$ be a Sylow $p$-subgroup of $G$ which contains $x_p$.  Then $C_P(x_p) \cap K$ is cyclic by hypothesis.  Let $P_1 = N_{P\cap K}(C_P(x_p) \cap K)$.  Then, as $\langle x_p\rangle$ is the unique subgroup of order $p$ in $C_{P}(x_p)\cap K$, $P_1$ centralizes $x_p$. Thus $P_1= C_P(x_p) \cap K$ which means that $C_P(x_p) \cap K= P \cap K$. Since $P\cap K$ is normal in $P$, we  have $x_p \in Z(P)$. By Sylow's Theorem, there exists $g \in G$ such that $f^g \in P$. Hence $x\sim x_p \sim f^g \sim a^g$ and this proves the claim.
\end{proof}

We now describe the relationship between the commuting graph and the prime graph of $G$.
For a connected component $\Psi$ of $\Gamma$ we write $\pi(\Psi)$ for the set which consists of the primes which divide the order of some element of $ \Psi$.
  The next theorem was proved by Iranmanesh and  Jafarzadeh \cite[Lemma 4.1]{I} in the special case that $\Gamma$ is connected.

\begin{theorem}\label{primegraph}  Suppose that $G$ is a finite group with $Z(G)=1$. Let $\Gamma$ be the commuting graph of $G$ and $\pi$ be the prime graph of $G$.  Let $ \Gamma/G$ be a set of representatives of $G$-classes of connected  components of $\Gamma$. Then the map $$\Psi \mapsto \pi(\Psi)$$ is a bijection between $\Gamma/G$ and the connected components of $\pi$. Furthermore, $\Gamma$ is connected if and only if $\pi$ is connected.

\end{theorem}

\begin{proof}
 We prove the theorem through a short series of three claims. We begin by showing that the map is well-defined.

\begin{claim}\label{primegraph1} Suppose $\Psi$ is a connected component of $\Gamma $. Then $\pi(\Psi)$  is a connected component of $\pi(G)$. Furthermore, if $\Psi$ and $\Psi^*$ are $G$-conjugate, then $\pi(\Psi)=\pi(\Psi^*)$.
\end{claim}

\medskip

  Set $\psi = \pi(\Psi)$. Suppose that $r$ and $s$  are in $\psi$  and let $x$ and $y$ be elements of order $r$ and $s$ respectively with $x$ and $y$ vertices of $\Psi$.  Then, as $\Psi$ is connected, there exists a path $$x= x_1 \sim x_2\sim \dots \sim x_m = y$$ connecting $x$ and $y$ in $\Psi$. We may suppose that,  for $1 \le i \le m$,  $x_i$ has prime order $r_i$. Then $x_ix_{i+1}$ has order $r_ir_{i+1}$. Hence, if $r_i \ne r_{i+1}$, then  $r_i$ is joined to $r_{i+1}$ in $\pi$. Therefore, ignoring loops,  there is a path $$r_1 \sim  r_2 \sim \dots \sim r_m$$ in $\pi$. Since each $r_i \in \psi$,   we get that $r$ and $s$ are connected in $\psi$. Therefore $\psi$ is connected. Let $\psi_0 \supseteq \psi$ be the connected component of $\pi(G)$ which contains $\psi$.
Assume that $r \sim s $ is  an edge in $\psi_0$   with $s  \in \psi$. Then there exists an element $vw $ of order $rs$ in $ G$ with $v$ of order $r$ and $w$ of order $s$.   Since $s\in \psi$, there exist an element $v$ of order $s$ in $\Psi$. Let $ S \in \Syl_{s}(G)$ with $v \in S$, then, as $Z(G)=1$,  $S^\# \subseteq \Psi$. Hence we may as well suppose that $w\in S$ and that $w\in \Psi$. Since $u$ and $w$ commute, we have $r\in \psi$.  Therefore $\psi = \psi_0$. \hfill $\blacksquare$

\medskip

Now we prove that the map is onto.

\medskip

\begin{claim}\label{primegraph2}
Suppose that  $\psi$ is a connected component of $\pi(G)$. Then there is a connected component $\Psi$ of $\Gamma(G)$ such that $\pi(\Psi)=\psi$.
\end{claim}

 \medskip
Let $p$ be a prime in $\psi$, let $x \in G$ be an element of order $p$ and let $\Psi$ be the connected component of $\Gamma(G)$ containing $x$. By Lemma~\ref{primegraph1}, $\pi(\Psi)$ is the connected component of $\pi(G)$ which contains $p$, thus $\pi(\Psi)=\psi$. \hfill $\blacksquare$

\medskip

Finally we show that the map is injective.

\medskip

\begin{claim}\label{primegraph3}
Let $\Psi $ and $\Theta$ be connected components of $\Gamma(G)$ such that $\pi(\Psi)= \pi(\Theta) $. Then $\Psi$ and $\Theta$ are $G$-conjugate.
\end{claim}

 \medskip
Let $p\in \pi(\Psi) =\pi(\Theta)$. Then there are Sylow $p$-subgroups $P$ and $Q$ of $G$ and an element $g\in G$ such that $P^g=Q$, $P^\# \subseteq \Psi$ and $Q^\# \subseteq \Theta$. Therefore $\Psi^g=\Theta$.
\hfill $\blacksquare$

The three claims together establish the theorem.
\end{proof}

\section{Isolated subgroups}\label{isolatedsec}

 With Theorem~\ref{primegraph} at our disposal we may use theorems that have already been provided by Williams \cite{Williams}. Assume that $G$ is a group and $\Gamma=\Gamma(G)$. We begin by recalling  the following definition from \cite{Williams}.

 A  subgroup $H$  of $G$ is called \emph{isolated} if  for all elements $g \in G\setminus N_G(H)$,  $H  \cap  H^g =1$ and for all $  h \in H^\#$,  $C_G(h) \le    H$.

\begin{lemma}\label{isolated}Suppose that $\Psi$ is a connected component of $\Gamma$. Then $\Psi = H^\#$ for some subgroup $H$ of $G$ if and only if $H$ is isolated.
\end{lemma}
  \begin{proof}  This is straightforward from the definition.
  \end{proof}

 We also recall that for a prime $p$ a proper subgroup $M$ of $G$ is \emph{strongly $p$-embedded} so long as  $p$ divides $|M|$ but  $p$ does not divide $|M \cap M^g|$ for all $ g\in G \setminus M$. There is a close relationship between groups with a strongly $p$-embedded subgroup and groups with an isolated subgroup.

\begin{lemma}\label{strongly} Assume that $Z(G)=1$ and suppose  $\Psi$ be a connected component of $\Gamma$ and let $M= \mathrm{Stab}_G(\Psi)$. Then either $M=G$ or, for all $p\in \pi(\Psi)$,  $M$ is strongly $p$-embedded.
\end{lemma}
\begin{proof}
Let $M=\mathrm{Stab}_G(\Psi)$ and  assume that $M<G$. For $p\in \pi(\Psi)$, let  $x\in \Psi$ have order $p$  and $P$ be a Sylow $p$-subgroup of $M$ containing $x$. Then, as $Z(G)=1$,  $P^\# \subseteq \Psi$ and therefore $(P^h)^\# \subseteq \Psi$ for each $h\in M$. In particular, if $u \in M^\#$ has order divisible by $p$, then $u \in \Psi$.

Suppose that $g\in G \setminus M$ and assume that $p \mid |M^g \cap M|$. Let $ y \in (M \cap M^g)^\#$ be an element of order $p$, then the previous paragraph implies $y\in \Psi$. The same argument applied to $M^g = \mathrm{Stab}_G(\Psi^g)$ yields $y\in \Psi^g$. Hence $\Psi=\Psi^g$ which gives $g\in M$, a contradiction.
\end{proof}

 The next theorem, which depends on the classification of the finite simple groups,  shows that if $G$ has non-abelian Sylow $r$-subgroups for some odd prime $r$, then $r $ and $2$ are in the same connected component of $\pi(G)$.

\begin{theorem}[Chigira, Iiyori \&  Yamaki, 2000]\label{CIY} Let $G$ be a finite group of even order and let $S$ be a non-abelian Sylow $r$-subgroup for some odd prime $r$. Then there exists $r\in S^\#$ such that $|C_G(r)|$ is even.
\end{theorem}
\begin{proof}
This is the main theorem in \cite{CIY}.
\end{proof}

 The following result is a minor modification of  \cite[Theorem 3]{Williams} which we obtain by using Theorem~\ref{CIY}. However, we remark that the statement seems to be implicit in the calculations in Williams' article.

\begin{theorem}[Williams, 1980] \label{Williams}  Suppose that   $G$  is  non-soluble  group with $Z(G)=1$, $\Psi$ is a connected component of $\Gamma(G)$
   and    $\psi=\pi(\Psi)$. Assume that $\psi$ does not contain $2$.  Then $ G$ has  an abelian  Hall  $\psi$-subgroup  $H$ which
is  isolated  in  $G$ and $\Psi= H^\#$. In particular, $\Psi$ has diameter 1.
\end{theorem}

\begin{proof}
Let $\psi = \pi(\Psi)$. Then $\psi$ is a connected component of $\pi(G)$ by Theorem~\ref{primegraph}.
By Williams \cite[Theorem 3]{Williams}, there exists  a  Hall $\psi$-subgroup $H$ of $G$ which is both isolated and nilpotent. Observe that $\Phi=H^\#$ is a component of $\Gamma$ and $\pi(\Phi)=\pi(H)=\psi=\pi(\Psi)$. Hence Theorem~\ref{primegraph} implies that $\Psi=\Phi^g$ for some $g\in G$, and so letting $K=H^g$ we see that $\Psi=K^\#$.
Now let $p\in \psi$ and pick $P$ a Sylow $p$-subgroup of $H$. Since no element of $P^\#$ can  be centralised by an involution, Theorem \ref{CIY} implies that $P$ is abelian. Thus $H$ is abelian, and so $\Psi$ has diameter 1.
\end{proof}

 Theorem~\ref{Williams} gives us a clear picture of disconnected commuting graphs of non-soluble groups. They consist of   perhaps more than one connected component containing involutions and then all the remaining connected components are cliques.  In simple groups we can say even  more as groups with strongly $p$-embedded subgroups have been investigated via the classification of simple groups when $p$ is odd and are known because of the fundamental work of Bender and Suzuki  \cite{Bender} when $p=2$.

\begin{lemma}\label{strongly2} Suppose that $G$ is a non-abelian simple group and $I<G$ is an isolated subgroup of $G$. Then either  \begin{enumerate} \item $I$ is cyclic of odd order;  or \item $I$ is a Sylow $r$-subgroup of $G$ and  $(G,r)$ is one of $(\PSL_2(r^a),r)$, $({}^2\mathrm B_2(2^a),2)$ with  $a$ odd or $(\PSL_3(4),3)$.\end{enumerate}
\end{lemma}
\begin{proof} Suppose that $I$ is an isolated subgroup with $I<G$ and assume that (i) does not hold.  Lemmas~\ref{isolated} and \ref{strongly} imply $N_G(I)$ is a strongly $r$-embedded subgroup of $G$  for each prime $r$ dividing $|I|$. If $|I|$ is even, then  we may suppose that $r=2$ and Bender's Theorem \cite{Bender} says that the pair   $(G,2)$ is one of   $(\PSL_2(2^a),2)$, $(\PSU_3(2^a),2)$,  $({}^2\B_2(2^a),2)$ and $I$ is a Sylow $2$-subgroup of $G$. Since the Sylow 2-subgroups of $\PSU_3(2^a)$, $a \ge 2$ are not isolated the statement in (ii) holds in this case. Hence we may assume that $I$ has odd order. By Theorem~\ref{Williams},  $I$ is abelian. In particular, if (i) is not true,   then $I$ is not cyclic and  there is a prime $r$ dividing $|I|$  such that $I$ has an elementary abelian subgroup of order $r^2$.  Therefore we may apply \cite[Theorem 7.6.1]{GLS3}  which yields the following possibilities for the pairs $(G,r)$: $(\PSL_2(r^a),r)$, $(\PSU_3(r^a),r)$,    $({}^2\mathrm G_2(3^a), 3)$, $(\mathrm{Alt}(2r),r)$, $(\PSL_3(4),3)$, $(\Mat(11),3)$, $({}^2\mathrm F_4(2)',5)$, $(\McL,5)$, $(\Fi_{22},5)$ or $(\J_4,11)$.  Furthermore, in each case $I$ is an abelian Sylow $r$-subgroup of $G$.   Thus $(\PSU_3(r^a),r)$  and $({}^2\mathrm G_2(3^a), 3)$ are not possible.   In $ \mathrm{Alt}(2r) $ elements of order $2r$, in $ \Mat(11)  $ there are elements of order 6, in  $ {}^2\mathrm F_4(2)$ there are elements of order 10, in $ \McL,5 $ there are elements of order 10, in $(\Fi_{22},5)$ there are elements of order 10 and  in $(\J_4,11)$ there are elements of order 22  and so in these cases $I$ is not isolated. This leaves the groups listed in the statement of (ii) and completes the lemma.
\end{proof}

\section{The reduction to almost simple}\label{almostsec}

For the remainder of this paper we assume that  $G$ is a   group with $Z(G)=1$ and we put $\Gamma=\Gamma(G)$.

\begin{lemma}\label{EGnot1}
Suppose that $E(G)\neq 1$ and $E(G)$ is not a quasisimple group. Then $\Gamma$ is connected and the diameter of $\Gamma$ is at most 7.
\end{lemma}
\begin{proof}
As $E(G)$ is a central product of at least two quasisimple groups,   $\Gamma^*=\Gamma(E(G))$ is connected and the diameter of $\Gamma^*$ is at most 3. By Thompson's Theorem \cite[Theorem 10.2.1]{Gor}, the elements of prime order in $G\setminus E(G)$ are connected to $\Gamma^*$. Hence every element of $G \setminus E(G)$ is connected to $\Gamma^*$ by a path of length at most $2$ and this gives the bound.
\end{proof}

\begin{lemma}\label{FarApartCloseToF}
Assume that $F(G) \neq 1$ and let $\mathcal Z = Z(F(G))^\#$. Suppose that $x,y\in\Gamma$ are such that $4< \d(x,y) < \infty$. Then $\d(x,\mathcal{Z})\leq 4$ and $\d(y,\mathcal{Z})\leq 4$.
\end{lemma}
\begin{proof} Assume   on the contrary that $\d(x,\mathcal{Z})>4$.
Let $Z=Z(F)$ and let $$\theta = x \sim x_1 \sim x_2 \sim x_3 \sim \dots \sim y$$ be a minimal path linking $x$ and $y$. For $i=1,2$ set $C_i=C_G(x_i)$ and note that, as $C_i \cap C_G(z)=1$ for all $z \in Z^\#$, $C_iZ$ is a Frobenius group by Lemma~\ref{Frob2}. If $|C_2|$ is odd, then we have $x_1 \sim x_3$ by Lemma~\ref{meta}, a contradiction as $\theta$ is a minimal path. Hence $|C_2|$ is even, and so, again by Lemma~\ref{meta}, there is a unique involution $t$ in $C_2$. Therefore  $t \in Z(C_2)$ and $[x_1, t]=1=[t, x_3]$. Since $t$ commutes with $x_1$  we get $t\in C_1$, whence $t \in Z(C_1)$  since $C_1$ is also a Frobenius complement. But now $x\sim t$ and so $x\sim t \sim x_3$ gives $3 = \d(x,x_3) \leq 2$, a contradiction. Hence $\d(x,\mathcal{Z})\le 4$ as claimed.
\end{proof}

\begin{lemma}\label{Fnot1}
Suppose  $F(G)\neq 1$ and let $\Psi$ be a connected component of $\Gamma$. Then the diameter of $\Psi$ is at most 9.
\end{lemma}
\begin{proof}
Suppose that $x,y \in \Psi$. Then $\d(x,y) <\infty$ so the previous lemma implies $\d(x,y) \leq 4$ or there are $f_1, f_2 \in Z(F(G))^\#$ such that $\d(x,f_1)\leq 4$ and $\d(y,f_1)\leq 4$. Since $Z(G)=1$ we have $f_1\sim f_2$ which gives $\d(x,y)\leq 4+1+4=9$.
\end{proof}

\begin{theorem}\label{almost}
Let $G$ be  a finite group with $Z(G)=1$ and $F^*(G)$ not a non-abelian simple group.  Then every connected component of $\Gamma$ has diameter at most $9$.
\end{theorem}
\begin{proof}
The hypothesis implies one of Lemmas~\ref{EGnot1} and \ref{Fnot1} is applicable.
\end{proof}

\section{Almost simple groups of sporadic type}\label{sporadic}
\label{sporsec}
Throughout this section $G$ is a group with $K=F^*(G)$ a sporadic simple group and we set $\Gamma=\Gamma(G)$.
Define $$\mathcal{D}_5=\{ \mathrm M_{11}, \mathrm M_{12},\mathrm{M}_{22}, \mathrm{Co}_2\},$$
$$\mathcal{D}_7=\{\mathrm{M}_{23},\mathrm M_{24}, \mathrm{Th}, \mathrm{H}\mathrm N, \mathrm{O}'\mathrm N, \mathrm{McL}, \mathrm{Ly}, \mathrm{J}_1,\mathrm{J}_{2},\mathrm{J}_4,\mathrm{HS}, \mathrm{Ru},\mathrm{Suz},\mathrm{Co}_1,\mathrm{Co}_3,\mathrm{Fi}_{22},\mathrm{Fi}_{23},\mathrm{Fi}_{24}',\mathrm B \},$$
$$\mathcal{D}_9=\{ \mathrm J_3, \mathrm M, \mathrm{He} \}.$$

Because of Theorem~\ref{Williams} we only need to focus our attention on the connected components of $\Gamma$ which contain involutions.

\begin{lemma}\label{G2}
There is a unique connected component of $\Gamma$ which has elements of even order.
\end{lemma}
\begin{proof}
Suppose on the contrary that there are distinct connected components $\Psi$ and $\Phi$ both containing involutions and let $x\in \Psi$ and $y\in \Phi$ be such. If $x$ is not conjugate to $y$, then Lemma \ref{involutionsgood} implies that $x$ and $y$ are connected, a contradiction. Hence $x$ and $y$ are conjugate, and so there is $g\in G$ such that $\Psi^g=\Phi$. Thus $\mathrm{Stab}_G(\Psi)\neq G$. Lemma \ref{strongly} now yields $\mathrm{Stab}_G(\Psi)$ is strongly 2-embedded in $G$. The fundamental theorem on strongly 2-embedded subgroups due to Bender and Suzuki  \cite{Bender}  provides the final contradiction.
\end{proof}

\begin{lemma}\label{3sp}
Suppose that $x$ and $y$ are involutions in $G$. Then $\d(x,y) \leq 3$.
\end{lemma}
\begin{proof}
If there are at least two conjugacy classes of involutions in $G$ this follows from Lemma~\ref{involutionsgood}. Otherwise, $G$ is one of $\mathrm M_{11}$, $\mathrm J_1$, $\mathrm M_{22}$, $\mathrm{M}_{23}$, $\mathrm{J}_3$, $\mathrm{McL}$, $\mathrm O$'$\mathrm N$, $\mathrm{Ly}$, $\mathrm{Th}$ by \cite[Table 5.3]{GLS3}. Since the commuting graph on a   conjugacy class of involutions in $G$ is a subgraph of $\Gamma$, we can apply \cite[Theorem 1.1]{BBPR2}   to obtain the required conclusion.
\end{proof}

\begin{theorem}
\label{Spor}
Suppose that $G$ is a group with $F^*(G)$ a sporadic simple group. Then the diameter of a connected component of  $\Gamma$ is at most $7$.
\end{theorem}
\begin{proof} By Theorem~\ref{Williams} and Lemma~\ref{G2} we only need to bound the diameter of $\Gamma_2$; the connected component of $\Gamma$ which contains an involution. Set $K= F^*(G)$ and let $\mathcal I$ be the set of involutions in $G$.

Suppose that $K\in \mathcal D_d$ with $d=5,7$ and let $x\in\Gamma_2$ be arbitrary. It suffices to show that $\d(x,\mathcal{I})\leq \frac{d-3}{2}$ (note that our assumption forbids $\langle x\rangle$ being self-centralising of prime order). If $x$ has even order then $\d(x,\mathcal{I})\leq 1$, so we can assume that $x$ has odd order. In particular, since $|\mathrm{Aut}(G):K|\leq 2$, we may assume $x\in K$.

   We   now make extensive use of the character tables available in \cite{ATLAS}. If $d=5$ we just check that $|C_{K}(x)|$ is even and if $d=7$ we have to check that $|C_{K}(x')|$ is even for each $x'\in K$ of prime order.

Suppose now that $d=9$ and $K\in \mathcal D_d$. We find then that $|C_G(x')|$ is even for each $x'\in K$ of prime order, unless $K\cong \mathrm J_3$ and $x'$  is in class $3B$, $K \cong \mathrm{He}$ and $x'$ is in the $7C$ class or $K\cong \mathrm{M}$ and $x'$ is in the $29A$ class (here we have used Atlas \cite{ATLAS} notation for conjugacy classes).
Notice that because in each of these cases there is a unique conjugacy class of elements with odd centralizer, the only candidates for $x$ which have distance $3$ from $\mathcal I$ must have  order a proper power of the order of $x'$. It follows at once from the character tables of these groups that $K= \mathrm J_3$ and $x$ is in class $9A$. However if $x$ is such an element, it also commutes with a $3$-central element in class $3A$ and this element commutes with an involution. Hence $\d(x,\mathcal I) \le 2$ for all $x \in \Gamma_2$ in this case.  Thus $\Gamma_2$ has diameter at most $7$ for $F^*(G)$ a sporadic simple group.
\end{proof}

\section{Almost simple groups of Alternating type}\label{alternating}

Here we assume that $G$ is a group with $Z(G)=1$ and $F^*(G)$ is isomorphic to $\mathrm{Alt}(n)$ for some $n\geq 5$. In \cite{I} it is shown that the diameter of $\Gamma$ is at most $5$ whenever $\Gamma=\Gamma(G)$ is connected. Here we settle for a short proof which provides a worse bound, but  also applies  to the case where $\Gamma$ is disconnected.


\begin{theorem}
\label{Alt}
Suppose that $n\geq 8$. Then the diameter of a connected component of $\Gamma(G)$ is at most 7 if $n=8$ or $n\geq 12$ and the diameter of a connected component of $\Gamma(G)$ is at most 8 if $9\leq n \leq 11$.
\end{theorem}
\begin{proof}
Let $\mathcal I$ be the set of involutions in $G$. Since there are at least two conjugacy classes of involutions, we find that $\mathcal I$ is a connected subset of $\Gamma(G)$ of diameter at most 3 and we let $\Gamma_2$ be the connected component of $\Gamma(G)$ containing $\mathcal I$. By the theorem of Williams it remains to show that $\Gamma_2$ has small diameter. Let $\mathcal{N}$ be the set of elements of order nine in $G$. We will prove that $\d(x,\mathcal I)\leq 2$ for every element of odd order in $\Gamma_2$, unless $x \in \mathcal N$ and $9\leq n \leq 11$, in which case $\d(x,\mathcal I)\leq 3$ holds. This will show that $\d(x,y)\leq 2+3+2=7$ if $n=8$ or $n\geq 12$ and in the remaining case, we have $ \d(x,y)\leq 3+2+2=8$ unless $x, y\in \mathcal{N}$ where we show that $\d(x,y)\leq 6$ holds. This gives the bound in the theorem.

Let $x\in \Gamma_2$ be an element of odd order and suppose that $\d(x,\mathcal I)> 2$. Write $x$ as a product of disjoint cycles $c_1,\dots,c_k$ and note that $x \sim c_i$ for $i=1,\dots,k$. If any of the $c_i$ fixes four or more points, then there is an involution conjugate to $(1,2)(3,4)$ in $G$ with which $c_i$ commutes, contradicting $\d(x,\mathcal I)>2$. Hence $x$ must be a single cycle which fixes no more than three points. Since $n\geq 8$ we see that 3-cycles commute with involutions, so $x$ fixes at most two points if $G=\mathrm{Alt}(n)$ and at most one point if $G=\mathrm{Sym}(n)$.

Choose $p$ to be the smallest prime dividing the order of $x$ and let $r$ be the order of $x$ divided by $p$. Then $x\sim x^r$ which has cycle type $1^{n-pr}p^r$. Note that $r\neq 1$ or we see that $\langle x \rangle =\langle x^r \rangle$ is a Sylow $p$-subgroup of $G$ and is isolated in $G$, which contradicts $x\in \Gamma_2$. Observe that the centraliser in $G$ of $x^r$ contains $C_p \wr \mathrm{Alt}({r})$. Since $\d(x,\mathcal I)>2$ this centraliser has odd order, and so $r=3$ (as $r$ is odd). Now $r$ is also prime and divides the order of $x$, so this implies $p=3=r$ and we have $x\in \mathcal N$. Observe then that $\d(x,\mathcal I)=3$ since $x^r$ commutes with a 3-cycle and as mentioned above, 3-cycles commute with involutions.

It remains to show that $\d(x,y)\leq 8$ for $x,y\in \mathcal N$ when $n=9,10,11$.  Write $x^3=u_1u_2u_3$ and $y^3=w_1w_2w_3$ where the $u_i$ and $w_i$ are 3-cycles. Since two 3-cycles either  commute with each other or another 3-cycle we see $\d(u_1,w_1)\leq 2$ and therefore $\d(x,y)\leq\d(x,u_1)+2+\d(w_1,y)=6$, as required.
\end{proof}

\begin{lemma}\label{smallcasesAn}
Suppose that $n\leq 7$. Then the maximum diameter of a connected component  of $\Gamma$ is given in Table \ref{tab:alts}.
\end{lemma}
\begin{proof}
For these values of $n$,  {\sc Magma} \cite{Magma} returns the answer fairly promptly, or one can use the character tables to reason as in  Section \ref{sporsec}.
\end{proof}

\begin{table}[h]
\centering
\begin{tabular}{|c |c |}
\hline
$G$ 			&		Maximum diameter of a component of $\Gamma$ \\ \hline
$\mathrm {Alt}(5)$	&		1 \\
$\mathrm {Sym}(5)$	&		5 \\
$\mathrm {Alt}(6)$	&		6 \\
$\mathrm {Sym}(6)$	&		4 \\
$\mathrm M_{10}$	&	6 \\
$\mathrm{PGL}_2(9)$	&	5 \\
$\mathrm{Aut}(\mathrm {Alt}(6))$	&	4 \\
$\mathrm {Alt}(7)$	&		5 \\
$\mathrm {Sym(7)}$	&		5 \\ \hline
\end{tabular}
\medskip
\caption{Maximum diameters of connected components for some almost simple groups.}
\label{tab:alts}
\end{table}

\section{Almost simple groups of Lie type}\label{Lie}

In this section we suppose that $K=F^*(G)$ is a finite simple group of Lie type defined in characteristic $r$. We may and do therefore assume that $G$ is a subgroup of $\mathrm{Aut}(K)$. The structure of $\mathrm{Aut}(K)/K$ is well understood and is revealed by \cite[Theorem 2.5.12]{GLS3}, which we  use without reference. We let $K^*$ be the subgroup of $G$ generated by all elements that act by conjugation  as inner-diagonal automorphisms on $K$.

For $p$ a prime which divides $|G|$,   let $\Gamma_p$ be a connected component of $G$ which contains an element of order $p$. Because of Theorem~\ref{Williams} our primary focus in on $\Gamma_2$ as if $\Gamma_p$ does not contain an involution, then $\Gamma_p$ has diameter $1$.

Our first  results examine  the commuting graphs in the small rank groups for which our generic argument fails.

\begin{lemma}\label{PSLPsi}
Suppose that $G \cong \PGL_2(r^a)$ or $\PSL_2(r^a)$ with $r$ odd. Then $\Gamma_2$ has diameter at most $5$ or $5<r^a\le 13$ in which case $\Gamma_2$ has diameter at most $6$. Furthermore every connected component in $\Gamma(\PSL_2(5))$ has diameter $1$.
\end{lemma}

\begin{proof}  If $G \cong \PGL_2(r^a)$, then  $G$ has two conjugacy classes of involutions. Furthermore, every element of $G$ other than the elements of order $r$ is connected to an involution. Thus, as the Sylow $r$-subgroups of $G$ are isolated, Lemma~\ref{involutionsgood} implies that $\Gamma_2$ has diameter at most $5$.

Suppose that $G \cong \PSL_2(r^a)$. If $r^a < 13$ an elementary computer calculation \cite{Magma}  reveals the stated result ($6$ being attained for $r^a=13$ and $r^a=9$). If $r^a>13$, then \cite[Theorem 1.1 (ii) and (iii)]{BBPR} asserts that in the commuting involution graph, two involutions have distance at most $3$ apart. Since the subgroups of order $(r^a+\epsilon)/2$, $\epsilon \equiv r^a \pmod 4$ are isolated in $\PSL_2(r^a)$ (see \cite[Table 1d]{Williams}),  $\Gamma_2$ consists of elements of order dividing $(r^a-\epsilon)/2$ and so every elements of $\Gamma_2$ has distance at most $1$ from an involution. Therefore $\Gamma_2$ has diameter at most $5$.
\end{proof}
Facts about ${}^2\mathrm G_2(3^a)$ are taken from  \cite[Theorem]{Ward}. We use them without specific reference in the next lemma.
\begin{lemma} \label{smallree} Suppose  $G=K\cong \mathrm {}^2\mathrm G_2(3^a)$ with $a>1$ odd. Then $\Gamma_3= \Gamma_2$ has diameter at most $8$.

\end{lemma}

\begin{proof} Let $\mathcal T$ be the set of  elements of order $3$ in $G$, $T $ be a Sylow $3$-subgroup  of $G$ and let $i \in N_G(T)$ be an involution.  Note that every element of order $3$ in $T$ is contained in $T'$ which itself is elementary abelian of order $3^{2a}$. We also know  $C_G(i) \cong 2 \times \PSL_2(3^a)$ and so $i$ centralizes subgroups of order $3^a$ in exactly $3^a+1$ Sylow $3$-subgroups of $G$.  Since  $T$ acts regularly on $\Syl_3(G) \setminus \{T\}$,  $|C_T(i)| = 3^a$ and $N_T(C_T(i)\langle i \rangle)= C_T(i)$, we obtain that if $S \in \syl_3(G) \setminus \{T\}$, then there is an involution $j \in N_G(T)$ such that $C_S(j) \ne 1$.  Let  $x, y \in \mathcal T $. Then $x \in U'$ and $y \in V'$ for some $U, V \in \syl_3(G)$. There is a unique involution $j$ which normalizes both $U$ and $V$ and we have  $C_U(j) $ and $C_V(j)$ are non-trivial. It follows that  $\d(x,y) \le 4$. Moreover,  this yields $\Gamma_3=\Gamma_2$ contains all elements of even order and all elements of order divisible by $3$.

 Now suppose that $x \in \Gamma_3$. If $C_G(x)$ has order divisible by $3$ or 2, then $\d(x,\mathcal T) \le 2$.  So suppose that $C_G(x)$ has odd order coprime to $3$.

 Recall that $$|G|= 3^{3a}(3^{3a}+1)(3^a-1)= 3^{3a} (3^a-1)(3^a+1)(3^a+3^{\frac{{a+1}}{2}}+1)(3^a-3^{\frac{{a+1}}{2}}+1)$$
Suppose that $p$ is an odd prime greater that $3$. Then $p$ divides exactly one of $(3^a-1)$, $(3^a+1)$, $(3^a+3^{\frac{{a+1}}{2}}+1)$ or  $(3^a-3^{\frac{{a+1}}{2}}+1)$.
If $x$ has order dividing $3^a+1$ or $3^a-1$, then $x$ can be seen the centralizer of an involution, which is a contradiction. So $x$  has order dividing $3^a\pm 3^{\frac{{a+1}}{2}}+1$. By \cite[Theorem (4)]{Ward}, $G$ has a Hall subgroup of order $3^a\pm 3^{\frac{{a+1}}{2}}+1$ and these subgroups are isolated.  Since $x$ is not contained in any such subgroup, this shows that $\Gamma_2=\Gamma_3$ has diameter at most 8.
\end{proof}

\begin{lemma}\label{rankone} Suppose $K\cong \PSL_2(r^a)$,  ${}^2\B_2(2^a)$ with $a\ge 3$ odd,  or ${}^2\mathrm G_2(3^a)$ with $a \ge 3$ odd.
Then $\Gamma_r$ has diameter at most $10$.
Furthermore, if $G$ is not $\PSL_2(r^a)$, $\PGL_2(r^a)$ with $r$ odd, then $\Gamma_r= \Gamma_2$.  \end{lemma}

\begin{proof}  If $G\cong  \PSL_2(r^a)$, $\PGL_2(r^a)$ or ${}^2\B_2(2^a)$, then  the Sylow $r$-subgroups of $G$ are  isolated and consequently $\Gamma_r$ has diameter one in the first two cases and two in the case of the Suzuki groups. We also note that in the first two cases, when $r$ is odd,  we have $2 \not \in \pi(\Gamma_r)$.   If $G =K \cong {}^2\mathrm G_2(3^a)$, then $\Gamma_3$ has diameter at most $8$ by Lemma~\ref{smallree}.
Thus we may assume that $G$ is not  isomorphic to $ \PSL_2(r^a)$, $\PGL_2(r^a)$,  ${}^2\B_2(2^a)$ or ${}^2\mathrm G_2(3^a)$. Especially we may assume that $G> K^*$.

 If there are no elements of prime order in $G \setminus K^*$, then$K \cong \PSL_2(r^a)$ with $r$ odd and $a$ even and  $\Gamma_r$  has diameter at most $2$ more than the diameter of a connected component of $\Gamma(K^*)$ containing an $r$-element. Thus  $\Gamma_r$ has diameter at most $8$ by Lemmas~\ref{primeelements} and \ref{PSLPsi}. Hence $G \setminus K^*$ contains  an element  $f$  of prime order $p$. Since $K^*$ has no graph automorphisms, $f$ must be a field automorphism. Therefore $C_K(f)  $ is a group of the same type as $K$ but defined over the field of order $q^{1/p}$ by Lemma~\ref{lem:lietypecents} (i). In particular, both $2$ and  $r$ divides  $|C_K(f)|$ and so $\Gamma_r=\Gamma_2$.

\begin{claim}\label{4} Any two $r$-elements of $K^*$ are at most distance $4$ apart in $\Gamma_r$.\end{claim}

\medskip

Let $x$ and $y$ have order $r$ and let $R_x, R_y \in \syl_r(K)$ with $x\in R_x$ and $y \in R_y$. Since $R_x$ is an $r$-group, if $R_x=R_y$, then $\d(x,y)\le 2$. So suppose $R_x \ne R_y$.  Since $K$ acts two transitively on  $\Syl_r(K)$ by conjugation, there exists $f \in G\setminus K^*$ of order $p$ such that
$|C_{R_x}(f)|= |C_{R_y}(f)| \ne 1$.  Now there exists $x_1 \in C_{R_x}(f) \cap Z(R_x)$, $y_1 \in C_{R_y}(f)\cap Z(R_y)$,  such that $$x \sim x_1 \sim f \sim y_1 \sim y$$ so $\d(x,y) \leq 4$ as claimed.\hfill$\blacksquare$

\medskip

We   intend to show that

\begin{claim} \label{3}every  element of $\Gamma_r$  has distance at most $3$  from an $r$-element unless $K^* \cong \PGL_2(r)$ and $x \in K^*\setminus K$ has order   $2m>2$ with $m$ odd, in which case $x$ has distance at most  $4$ from an $r$-element and has even order.  \end{claim}

\medskip

To prove this we only have to show that such elements have distance $2$ from a field automorphism unless the exceptional case occurs.

  Assume that  $x \in K^*\cap  \Gamma_r$.
We may as well suppose that $x$ is an $r'$-element as otherwise $x$ is connected to an $r$-element.  Thus $x$ is contained in a torus $T$ of $K^*$.

Suppose   $x$ has even order. Then $x $ is incident to an $i$ involution. If $i \in K$, then it centralized by a field automorphism of $G$ and we have our claim. If $K$ is a Suzuki-Ree group or if $K ^* = \PSL_2(r^a)$,  we do not need to work further.  If $K^* \cong \PGL_2(r^a)$ and $i \in K^*\setminus K$, then $i$ is joined to an involution $j \in K$ and this is
the exceptional configuration.

Suppose $x$ has odd order.  If some power of $x$ which is not the identity, commutes with an $r$-element, we are done.  So suppose that no power of $x$ commutes with such an element.
Let $y$ be  power of $x$ such that $y$ has prime order $p$. Then  Lemma~\ref{lem:lietypecents} (c) implies that $C_{K^*}(y) \ge T$ and $C_{K^*}(y)/T$ is an elementary abelian $p$-group which is  isomorphic to a subgroup of $Z$ the centre of the universal version   of $K^*$. Since  in all cases under investigation  $Z$ has order at most $2$, we have  $C_{K^*}(y)=C_{K^*}(x)=T$. Since $T$ is abelian, any element of $K^*$ which is joined to $y$ is also joined to $x$.  Hence any shortest path to an $r$-element from $x$ must go via  a field automorphism $f$ of prime order such that $C_T(f) \ne 1$. But then $x$ has distance at most $2$ from $f$. Thus our claim is established if $x \in K^* \cap  \Gamma_r$.

Now suppose that $x \in G \setminus K^*$.   Let $p$ be a prime that divides the order of $xK^*$, $\langle x^*\rangle \in \syl_p(  \langle x \rangle)$ and $y \in \langle x\rangle$ have  order $p$. Then $x \sim y$. If $y \not \in K^*$, then $y$ is a field automorphism by \cite[Theorem 4.9.1 (d)]{GLS3} and this means $x$ has distance at most $2$ from an $r$-element.

So suppose that $y \in K^*$. We may as well assume that $y$ is not an $r$-element.  If $y$ has order $2$, then $y$ commutes with a field automorphism of $K$ and so $x$ has distance $3$ from an $r$-element. So $y$ has order $p$ different from $2$ and $r$. Since $p$ is odd, $K^*\langle x \rangle = K^*\langle f\rangle$ for some field automorphism $f$ of $K$.\footnote{The only danger case arises if $K^*\cong \PSL_2(r^a)$ with $r$ odd, $xK$ has order $2$ and is a product of an  outer diagonal and a field automorphism. In this case $xK$ contains no elements of order $2$.}
Moreover the Sylow $p$-subgroups of $K^*$ are abelian and if $P \in \syl_p(K^*\langle x\rangle)$ with $x^* \in P$. Then $P = \langle x^*\rangle T$ where $T$ is an abelian subgroup of $K^*$. It follows that $y \in Z(P)$. Since $P$ contains a field automorphism, once again we have $y$ is incident to a field automorphism. This means that $x$ has distance at most $3$ from an $r$-element. This completes our verification of \ref{3}.\hfill$\blacksquare$

\medskip

Combining \ref{4} and \ref{3} yields our assertion that  the diameter of $\Gamma_r$ is at most $10$ so long as $K^* \ncong \PGL_2(r^a)$. So suppose that $K^* \cong \PGL_2(r^a)$ and that $x, y \in G$ have $\d(x,y) \ge 11$. Then we may assume that $x$ has distance $4$ from an $r$-element and is incident to an involution $i \in K^* $. Also $y$ has distance a least $3$ from an $r$-element. If $y$ has distance $4$, then $x$ and $y$ have even order by \ref{3}. But then $\d(x,y) \le 5$ by Lemma~\ref{involutionsgood}, which is a contradiction. Now $y $ has distance $2$ from a field automorphism. Thus $y$ has distance $3$ from an involution $j$ and therefore $\d(x,y) \le \d(x, i)+\d(y,j) + \d(i,j)\le 1+2+3=6$  again by Lemma~\ref{involutionsgood}.  This contradiction completes our proof.
\end{proof}

The bound in Lemma~\ref{rankone} is probably too high however it is not too far from reality. The following remark originates from a computer calculation.

\begin{remark} In the group ${}^2\B_2(2^5){:}5$ there are elements of order $25$ that have distance $8$ apart. In particular, the diameter of $\Gamma_2$ for this group is at least $8$.
\end{remark}

\begin{lemma}\label{Tits}  Suppose that $K={}^2\mathrm F_4(2)'$. Then $\Gamma_2$ has diameter at most $5$.
\end{lemma}
\begin{proof} We use the \cite[page 74]{ATLAS}. This shows that $G$ has two conjugacy classes of involutions and so by Lemma~\ref{involutionsgood} any two involutions have distance at most $3$ apart.
 As $G$ has exactly one conjugacy class of elements of order $3$ and of order $5$ and they both have centralizers of even order, such elements are distance one from an involution as are all the elements of even order. Hence, as the elements of order $13$ generate isolated subgroups, we have $\Gamma_2$ has diameter at most $5$.
\end{proof}

\begin{lemma}\label{2F42}
Suppose that $K\cong {}^2\mathrm F_4(q)$ for $q>2$ an odd power of $2$.Then $\Gamma_2$ has diameter at most $9$.
\end{lemma}
\begin{proof}
By \cite[(18.2)]{aschseitz} $K$ has two conjugacy classes of involutions and by \cite[(18.6)]{aschseitz} the centralisers in $K$ of representatives from different classes are non-isomorphic. Hence $G$ has exactly   two conjugacy classes of involutions (since $|G:K|$ is odd here). Let $\mathcal I$ be the set of involutions in $G$.  Then Lemma \ref{involutionsgood} implies that any two members of $\mathcal I$ are at most distance  3 apart, all 2-elements are contained in $\Gamma_2$ and the distance between any pair of elements of even order in  $G^\#$ is at most 5. Thus it remains to show that $\d(x,\mathcal{I})\leq 3$ for any $x\in \Gamma_2$ of odd order.

Let $x\in G$ have odd order and assume $\d(x,\mathcal I)>3$.  Choose $t \in \mathcal I$ so that $\d(x,t)$ is minimal   and  let $$x \sim x_1 \sim x_2 \sim \dots \sim t$$ be a shortest path between $x$ and $t$  with the inner elements of prime order.
Since the only automorphisms of $K$ of prime order are inner and field automorphisms (see \cite[Theorem 2.5.12 and Definition 3.5.13]{GLS3}), if either $x_1$ or $x_2$ are in $G \setminus K$,   Lemma \ref{lem:lietypecents}(a) implies $C_K(x_1)$ or $C_K(x_2)$ has even order which contradicts the choice of $x$. Hence $x_1, x_2 \in K$. For $i=1,2$, set  $C_i= C_G(x_i)$.
By our choice of $t$,  $C_i$ has odd order. Thus Lemma \ref{lem:lietypecents} implies $C_i\cap K$ has a normal abelian subgroup $T_i$ of odd order and $|(C_i \cap K)/T_i|$ divides $|Z|$ where $Z$ is as in Lemma \ref{lem:lietypecents}.   By \cite[Theorem 2.2.9, Table 2.2]{GLS3},  $|Z|=1$. Therefore $C_i\cap K$ is abelian. Now, as $x_2 \in C_1 \cap K$ and $C_1 \cap K$ is abelian, we have $C_1 \cap K \le C_2 \cap K$. Similarly, $C_2 \cap K \le C_1 \cap K$. Hence $C_1\cap K = C_2 \cap K$ and,  furthermore, if $x \sim x_1^* \sim x_2$ with $x_1^*$ of prime order,  then   $C_K(x_1^*)= C_K(x_2)$. Finally notice that $C_1 \cap K$ is an isolated subgroup of $K$ and therefore Lemma~\ref{strongly2} implies  $C_1 \cap K$ is a cyclic group.

If $x \in K$, then $x,x_2 \in C_1 \cap K$ and $x \sim x_2$, a contradiction.   Therefore $x \in G \setminus K$  has the property that no element of prime order in $\langle x \rangle$ is outside of $K$ for otherwise $x$ would centralize a field automorphism. On the other hand $K\langle x\rangle$ contains a $p$-element $f$ which centralizes an element of $\mathcal I$. Thus Lemma~\ref{outside2} implies $\d(x,\mathcal I) \le 3$. This contradiction concludes our proof.
\end{proof}

\begin{lemma} \label{psl34} Suppose  $K=K^* \cong \PSL_3(4)$. Then $\Gamma_2$ has diameter at most $5$.
\end{lemma}

\begin{proof}

 Let $\mathcal I$ be the set of involutions in $K$.  Suppose first that $G=K$. Since $C_K(i)$ is a $2$-group for $i \in \mathcal I$,    $\Gamma_2$ consists of elements of even order. By \cite[Theorem 1.2]{BBPR}, the commuting involution graph in $K$ is connected of diameter $3$. Hence $\Gamma_2$ has diameter $5$ in this case.

There are three extensions of $K$ by a subgroup of order $2$ and a unique extension by a group of order $4$. All these extensions split. It follows that if $G$ is one of these extensions, then  $G$ has at least two conjugacy classes of involutions and so by Lemma~\ref{involutionsgood} any two involutions are at distance at most $3$ apart in $\Gamma$.  By the first paragraph any element of odd order in $\Gamma_2$  must be connected to an involution in $G \setminus K$ and so $\Gamma_2$ has diameter at most $5$ in these cases.  Since $\mathrm{Out}(K) \cong \mathrm{Sym}(3) \times 2 $, this completes the proof.
\end{proof}

From here on we assume that  $K$ is not a Suzuki-Ree group and is not isomorphic to $\PSL_2(r^a)$. We define  $\mathcal{R}$ to be the set of long root elements of $K$. Notice that exclusion of the  Suzuki-Ree groups means that $\mathcal R$ is non-empty.

The need to remove $\PSL_2(r^a)$ and the necessity of Lemma~\ref{psl34} is highlighted by the following lemma which shows that  the distance between two members of $\mathcal{R}$ is  at most $2$ and allows us to direct our attention towards bounding $\d(x,\mathcal{R})$ for each $x\in \Gamma_r$.

\begin{lemma}\label{rootelements} Assume $K$ is not $\PSL_3(2)\cong \PSL_2(7)$ and if $K \cong \PSL_3(4)$ that $K^* \cong \PGL_3(4)$. Then the distance between any two elements of $\mathcal R$ in $\Gamma$ is at most $2$.
\end{lemma}

\begin{proof}
Let $u_1$ and $ u_2 $ be elements of $ \mathcal{R}$. By Lemma~\ref{rootgroups} (iv) either $\langle u_1, u_2\rangle$ is an $r$-group or $\langle u_1, u_2\rangle$ is conjugate to a subgroup of  a fundamental $\SL_2(r^a)$  which we call $S$.  Obviously if $\langle u_1,u_2\rangle$ is an $r$-group, then $\d(u_1,u_2) \le 2$ while, if $\langle u_1, u_2\rangle$ is conjugate to a subgroup of  $S$, then $\d(u_1,u_2)\le 2$ as, by Lemma~\ref{rootgroups} (v), $C_K(S) \ne 1$.
\end{proof}

The following fact implies $\d(x,\mathcal{R})\leq 2$ whenever $O^{r'}(C_K(x))\neq 1$ for $x\in K$.

\begin{lemma}
\label{rootsincentre}
Suppose that  $h$ is  an $r$-element of $K$. Then $h$ centralises a root element.
\end{lemma}
\begin{proof}
Let $U$ be a Sylow $r$-subgroup of $K$ containing $h$.  Then, by  \cite[Theorem 3.3.1]{GLS3} $Z(U)$ contains long root elements, which proves the lemma.
\end{proof}

\begin{prop}\label{concomp10} Suppose  $K$ is a finite simple group of Lie type and assume that  $K$ is not $\PSL_2(r^a)$, a Suzuki -Ree group, $\PSL_3(2)$ and, if  $K \cong \PSL_3(4)$ assume that $K^* \cong \PGL_3(4)$.  Let $u \in \mathcal R$, set $G^*= C_G(u)K$ and let $\Gamma_r$ be a connected component of $\Gamma$ which contains $\mathcal R$.  If $x \in \Gamma_r$, then $\d(x,\mathcal{R}) \leq 4$. In particular,    $\Gamma_r=\Gamma_2$ has diameter at most $10$.
\end{prop}

\begin{proof} Let $J$ be a fundamental $\SL_2$-subgroup of $G$. By Lemma~\ref{rootgroups},  $|G:N_G(\mathcal R)|= 2$ and $|N_G(\mathcal R):G^*|=2$. Hence, if $x \in G \setminus G^*$, then $x$ has even order. In particular, $x$ commutes with an involution $y$. If $r=2$, then $y$ has distance at most $2$ from an element of $\mathcal R$  and so $\d(x,\mathcal R)= 3$ in this case.  If $r$ is odd, then $Z(J)$ has an involution $z$. If $y$ and $z$ are conjugate then $\d(x,\mathcal R)= 2$ whereas, if $y$ and $z$ are not conjugate, then $\d(y,z)\le  2$ and so $\d(x,\mathcal R)\le 4$. Hence we may suppose that $x \in G^*$. If $x \not \in K$, then $u^{G^*}= u^K$ and so $\d(x,\mathcal R) \le 4$ by Lemma~\ref{outside}. Hence from now on we may assume that $x \in K$.

 Let $\Gamma_r^o$ be the connected component of $\Gamma(K^*)$ which contains $\mathcal R$. Then obviously we have $\Gamma_r \cap K^* \supseteq \Gamma_r^o$.  Just for precision we mention here that it is possible that $K^* \not \le G^*$.
Assume that $x \in (G^* \cap \Gamma_r)\setminus \Gamma_r^o$. Then   $x$ is contained in an abelian isolated subgroup $I$ of $K ^*$ by Theorem~\ref{Williams}. Since $x \in \Gamma_r$, there is an element $f$ of prime order in $G \setminus K^*$ such that $C_I(f) \ne 1$.  Since every element of prime order in $G\setminus K^*$ is a graph or a graph-field automorphism and  $I$ is abelian, $\d(x,f)= 2$ by Lemma~\ref{lem:lietypecents}.
Since $f$ centralizes an $r$-element, we have $\d(x,\mathcal R)\le 4$ by Lemma~\ref{rootsincentre}.    Finally we have to consider the case when $x \in \Gamma_r^o \cap K$.

%

Let $x \in \Gamma_r^o$ be an arbitrary element. If $x$ is not an $r'$-element, then plainly $x$ has distance at most 1 from an $r$-element. Hence we may suppose that $x$ is a semisimple element.  Let $u$ be an $r$ element of $\Gamma_r$ and $$\pi =(x \sim x_1 \sim x_2\sim \dots \sim u)$$ be a shortest path to $u$. Assume that the length of $\pi$ is at least $3$.   Let $x_1$ have order $p_1$ and $x_2$ have  order $p_2$ where $p_1$ and $p_2$ are  primes. Then $C=C_{K^*}(x_1)$ is  an $r'$-group and $C$ is certainly non-abelian for otherwise $x \sim x_2$, which contradicts the minimality of $\pi$.
 By \cite[Theorem 4.9.1]{GLS3}(g) and (h) we see that $x_1$ is an inner-diagonal automorphism of parabolic type. Furthermore Lemma~\ref{lem:lietypecents} shows that $C$ contains a normal, abelian  $r'$-subgroup $T$ such that $C/T$ is an elementary abelian $p_1$-subgroup which is isomorphic to a subgroup of $Z(K_u)$ where $K_u$ is the universal covering group of $K$.  Since $C$ is not abelian, we have $C/T\ne 1$.   Now, if $p_1=2$, then $r$ is odd. In this case, $Z(J)$ has order $2$  and is not conjugate to $\langle x_1\rangle$ and so $x_1$ is distance at most $2$ from $Z(J)$ and    distance at most $3$ from $\mathcal R$. Thus we may suppose that $x_1$ does not have order $2$.
In particular,  $|C/T|$ is  odd.  Since $C/T$ is isomorphic to a subgroup of $Z(K_u)$, \cite[Theorem 2.2.9, Table 2.2]{GLS3} implies $C/T$ is a cyclic group and one of the following holds.
 \begin{enumerate}
 \item $p_1=3$ divides $r^a-1$ and $K \cong\mathrm E_6(r^a)$;
 \item $p_1=3$ divides $r^a+1$ and $K \cong{}^2\mathrm E_6(r^a)$;
 \item $p_1 $ divides $(r^a-1,n)$ and $K \cong \PSL_n(r^a)$; or \item $p_1$ divides $(r^a+1,n)$ and $K \cong  \PSU_{n}(r^a)$.
  \end{enumerate}
If possibility (i) or (ii) holds, then $x_1$ has order $3$  and,  as $J$ has order divisible by $3$, $x_1$ is distance at most $2$ from an  element of a conjugate of  $J$. Since $J$ commutes with a long root element (contained in a subgroup isomorphic to a group of type $\mathrm A_5(r^a)$ in case (i) and in a subgroup isomorphic to a group of type ${}^2\mathrm A_5(r^a)$ in case (ii)), we have that $x$ has distance at most $4$ from $\mathcal R$  in these first two  cases.

Thus we may suppose that cases (iii) and (iv) pertain. In these cases $p_1$ divides $r^a-1$ in case (iii) and $r^a+1$ in case (iv)  and we know that $p_1$ is not $2$. In  particular we have that $p$ divides the order of $J$. If $n \ge 4$, then $J$ is centralized by a long root element and so we have $x_1$ is distance at most three  from a long root element. Hence $\d(x,\mathcal R) \le 4$.

Suppose   that $n=3$. Then $p_1=3$ as well.

First assume that $K\cong \PSL_3(r^a)$ and   put  $\wh {K^*}= \mathrm{GL}_3(r^a)$. Let $V$ be the natural $\mathrm{GL}_3(r^a)$-module and set $Z= Z(\wh{K^*})$. So $|Z|=(r^a-1)_3$. For $w \in K^*$,   let $\wh w$ be an element of minimal order in $\wh {K^*}$ with $\wh w Z  = w$.  Let $D \in \syl_3(C_{K^*}(x_1))$ and let $\wh D  $  be a $3$-group such that $\wh  D  Z/Z \ge  D$.   As $|K^*|_3 >  3$, $|C|_3 \ge 3^2$ and so   $$|\wh D|\ge |C|_3  (r^a-1)_3\ge  3^2(r^a-1)_3.$$ If $\wh D$ is cyclic, then  Schur's Lemma implies that $|D| \le (r^{3a}-1)_3= 3(r^a-1)_3$ which is a contradiction.
Hence $\wh D$ is not cyclic.
Therefore,  there exists $\wh{x_2^*} \in \wh D \setminus Z$  of order $3$. Let  $z \in Z$ have  order $3$. Then $C_V(\wh{x_2^*}z^j)\ne 0$ for some $1 \le 3 \le j$. But then $\wh{x_2^*}$ preserves the decomposition of $V$ into a sum of a $1$-dimensional and a $2$-dimensional subspace. So $\wh{x_2^*} \in \wh X \cong \mathrm{GL}_1(r^a)\times \mathrm{GL}_{2}(r^a)$. So $\wh{x_2^*}$ commutes with an element $\wh y \in Z(\wh X)$ and this element commutes with an element of $  \wh {\mathcal R}$. Now taking  images in $K^*$ and using the fact that, if $r^a=4$, then $K^*= \PGL_3(4)$, we have $\d(x,\mathcal R) \le 4$.

We argue similarly for $K \cong \PSU_3(r^a)$. So let $\wh {K^*} = \mathrm{GU}_3(r^a) $  and $V$ be the natural module for $\wh K^*$. We use all the conventions of the previous paragraph.
Let $D \in \syl_3(C_{K^*}(x_1))$ and let $\wh D  $ be a $3$-group such that  $\wh  D  Z/Z= D$.   Obviously  $|\wh D| = |C|_3(r^a+1)_3 \ge  3^2(r^a+1)_3$. If $\wh D$ is cyclic, then  Schur's Lemma implies that $|D| \le (r^{6a}-1)_3= 3(r^{2a}-1)_3=3(r^{a}+1)_3$ which is a contradiction.
Hence $\wh D$ is not cyclic.

Therefore there exists $\wh{x_2^*} \in \wh D \setminus Z$  of order $3$. Let  $z \in Z$ have  order $3$. Then $C_V(\wh{x_2^*}z^j)\ne 0$ for some $1 \le 3 \le j$. But then $\wh{x_2^*}$ preserves the decomposition of $V$ into a direct sum of $3$ perpendicular non-degenerate  $1$-dimensional spaces. In particular,  $\wh{x_2^*} \le X \cong \mathrm{GU}_1(r^a)\times \mathrm{GU}_{2}(r^a)$. Thus  $\wh{x_2^*}$ commutes with an element $\wh y \in Z(X)$ and this element commutes with an element of $  \wh {\mathcal R}$. Hence $\d(x,\mathcal R) \le 4$. So we have shown $\d(x,\mathcal R) \le 4$ in all cases.

Finally combining this with Lemma~\ref{rootelements} implies that $\Gamma_r=\Gamma_2$ has diameter at most $10$.
\end{proof}

\begin{remark}  Let $G= \mathrm {PGL}_3(7)$. Then $7^2+7+1= 57$ and, if $x$ is an element of order  of order $57$, then $\d(x, \mathcal R)= 4$.
\end{remark}

 We combine the results of this section in a single statement.

 \begin{theorem}\label{Lies} Suppose that $F^*(G)$ is a simple group of Lie type. Then every connected component of   $\Gamma$ has diameter at most 10.
 \end{theorem}

\begin{proof} We know all the connected components other than $\Gamma_2$ are of diameter one by Theorem~\ref{Williams}. Thus we only need to consider $\Gamma_2$.
 If $K$ is a Suzuki-Ree group, then Lemmas~\ref{rankone}, \ref{Tits} and \ref{2F42} provide the result. When $K \cong \PSL_2(r^a)$, we apply Lemma~\ref{rankone} and note also that $\PSL_2(7) \cong \PSL_3(2)$. If $K \cong \PSL_3(4)$ and $K^* \ne \PGL_3(4)$, then Lemma~\ref{psl34}  yields $\Gamma_2$ has diameter at most $5$. All the remaining groups are considered in Proposition~\ref{concomp10}.

\end{proof}

\section{The proof of Theorem~\ref{MainTheorem}}

We now draw all our strings together and prove Theorem~\ref{MainTheorem}. So suppose that $G$ is a finite group with trivial centre and let $\Gamma=\Gamma(G)$. Assume that $\Gamma$ has diameter at least $11$. Then Theorem~\ref{almost}  states that $F^*(G)$ is a non-abelian simple group.  Using Theorems~\ref{Spor} and \ref{Alt}, Lemma~\ref{smallcasesAn}  and the classification of finite simple groups, we have that $F^*(G)$ is a simple group of Lie type. Finally Theorem~\ref{Lies} provides the contradiction.

\end{document}